\documentclass[english,10pt]{amsart}

\usepackage[english]{babel}

\usepackage{amscd,amssymb,amsfonts,amsmath}
\usepackage{tikz-cd}

\usepackage{bm}
\usepackage{graphics}
\usepackage{epsfig}
\usepackage{mathrsfs}
\usepackage{xcolor}
\DeclareMathAlphabet{\mathscrbf}{OMS}{mdugm}{b}{n}

\definecolor{violet}{rgb}{0.0,0.2,0.7}
\definecolor{rouge2}{rgb}{0.8,0.0,0.2}
\usepackage{hyperref}
\hypersetup{
    unicode=false,          
    pdftoolbar=true,        
    pdfmenubar=true,        
    pdffitwindow=false,     
    pdfstartview={FitH},    
    pdftitle={},    
    pdfauthor={},     
    colorlinks=true,       
   linkcolor=violet,          
    citecolor=rouge2,        
    filecolor=black,      
    urlcolor=cyan}           

\usepackage[text={6.7in,9.2in},centering]{geometry}

\makeatletter
\renewcommand\subsection{\@startsection{subsection}{2}%
  \z@{.5\linespacing\@plus.7\linespacing}{-.5em}%
  {\normalfont\sffamily}}  
\makeatother

\renewcommand{\phi}{\varphi}

\newcommand{\wh}{\widehat}
\newcommand{\wb}{\overline}

\renewcommand{\le}{\leqslant}
\renewcommand{\ge}{\geqslant}

\newcommand{\sD}{\mathscr{D}}

\newcommand{\sI}{\mathscr{I}}

\newcommand{\sL}{\mathscr{L}}
\newcommand{\sM}{\mathscr{M}}

\newcommand{\sO}{\mathscr{O}}
\newcommand{\sP}{\mathscr{P}}
\newcommand{\sQ}{\mathscr{Q}}

\newcommand{\sX}{\mathscr{X}}
\newcommand{\sY}{\mathscr{Y}}

\newtheorem{thm}{Theorem}[section]

\newtheorem{lemma}[thm]{Lemma}
\newtheorem{cor}[thm]{Corollary}

\newtheorem{prop}[thm]{Proposition}

\newtheorem*{thm*}{Theorem}
\theoremstyle{definition}

\newtheorem{defn-thm}[thm]{Definition-Theorem} 
\newtheorem{defn-lemma}[thm]{Definition-Lemma}

\newtheorem{not-and-defs}[thm]{Notation and definitions}
\newtheorem{not-and-def}[thm]{Notation and definition}

\theoremstyle{remark}

\newtheorem{fact}[thm]{Fact}

\newtheorem{rem}[thm]{Remark}

\numberwithin{equation}{section}

\def\factor#1.#2.{\left. \raise 2pt\hbox{$#1$} \right/\hskip -2pt\raise -2pt\hbox{$#2$}}

\begin{document} 

\title[$A$-analyticity of separatricies of foliations]{$A$-analyticity of separatricies of foliations}

\author{St\'ephane \textsc{Druel}}

\address{St\'ephane Druel: Univ Lyon, CNRS, Universit\'e Claude Bernard Lyon 1, UMR 5208, Institut Camille Jordan, F-69622 Villeurbanne, France.} 

\email{stephane.druel@math.cnrs.fr}



\begin{abstract}
Let $X$ be a smooth quasi-projective surface over a number field $K$, and let $L$ be a foliation on $X$. We prove that if $L$ is closed under $p$-th powers for almost all primes $p$, then any $L$-invariant smooth formal curve is $A$-analytic. Building on prior work of Bost we obtain an algebraicity criterion for those curves.  
\end{abstract}

\maketitle
{\small\tableofcontents}

\section{Introduction}

Let $X$ be a smooth quasi-projective surface over a number field $K$, and let $L$ be a foliation on $X$ (defined over $K$). 

If $L$ is closed under $p$-th powers for almost all primes $p$, then the generalization to foliations by Ekedahl, Shepherd-Barron and Taylor (see \cite[Conjecture F]{esbt}) of the classical Grothendieck-Katz conjecture predicts that $L$ has algebraic leaves. The conjecture has been shown in some special cases, e.g. if $X$ is a abelian surface and $L$ is induced by a non-zero vector field (\cite[Theorem 2.3]{bost}), if $X$ is the total space of a line bundle over an algebraic curve and $L$ is induced by a linear connection (\cite[Corollaire 4.3.6]{andre}), and if $X$ is a $\mathbb{P}^1$-bundle over an elliptic curve $E$ and $L$ is an Ehresmann connection on $X \to E$ (\cite[Proposition 9.3]{cd1fzerocan}).

Actually, \cite[Theorem 2.3]{bost} follows from an algebraicity criterion for smooth formal schemes in algebraic varieties over number fields (\cite[Theorem 3.4]{bost}), which we recall now. Let $P \in X(K)$, and let $\wh{V}$ be a smooth formal subscheme of the completion $\wh{X}_P$ of $X$ at $P$. If $\wh{V}$ is A-analytic (we refer to Section \ref{section:notation} for this notion) and $\wh{V}_\mathbb{C}$ satisfies the Liouville condition for some embedding $K \subset \mathbb{C}$, then the formal scheme $\wh{V}$ is algebraic. 

If $L$ is regular at $P \in X(K)$, then the formal leaf $\wh{V}$ of $L$ through $P$ is smooth but $\wh{V}$ is not A-analytic in general. Nevertheless, if $L$ is closed under $p$-th powers for almost all primes $p$, then $\wh{V}$ is $A$-analytic by \cite[Proposition 3.9]{bost}.

Suppose now that $P\in X(K)$ be a singular point of $L$, and let $K \subset \mathbb{C}$ be an embedding. In \cite{camacho_sad}, Camacho and Sad proved that there is at least one separatrix of 
$L_\mathbb{C}$ through $P_\mathbb{C} \in X_\mathbb{C}$, $\textit{i.e.}$ a (local) irreducible complex curve in $X_\mathbb{C}$ passing through $P_\mathbb{C}$ and tangent to $L_\mathbb{C}$.
Suppose in addition that $P_\mathbb{C}$ is a non-degenerate reduced singularity of $L_\mathbb{C}$. Then, by \cite[Appendice II]{mattei_moussu}, $L_\mathbb{C}$ has exactly two separatrices through $P_\mathbb{C}$. They are smooth and intersect transversely at $P_\mathbb{C}$.
In this paper, we extend \cite[Proposition 3.9]{bost} to this setting.

\begin{thm}\label{thm_intro}
Let $X$ be a smooth quasi-projective surface over a number field $K$, and let $L$ be a foliation on $X$. Suppose that $L$ is closed under $p$-th powers for almost all primes $p$.
Let $P\in X(K)$ be a singular point of $L$. Suppose furthermore that $P_\mathbb{C}$ is a non-degenerate reduced singularity of $L_\mathbb{C}$ for some embedding $K \subset \mathbb{C}$. Then the following holds.
\begin{enumerate}
\item There exist smooth formal subschemes $\wh{V}$ and $\wh{W}$ over $K$ of the completion $\wh{X}_P$ of $X$ at $P$ such that $\wh{V}_\mathbb{C}$ and $\wh{W}_\mathbb{C}$ are the formal completion at $P$ of 
the separatrices of $L_\mathbb{C}$ through $P_\mathbb{C}$ for any embedding $K \subset \mathbb{C}$.
\item  The formal curves $\wh{V}$ and $\wh{W}$ are $A$-analytic.
\end{enumerate}
\end{thm}

In fact, a slightly more general statement is true (see Theorem \ref{thm:A-analyticity} and Corollary \ref{cor:cor}).

\medskip

The following is an immediate consequence of \cite[Theorem 3.4]{bost} and Theorem \ref{thm_intro} together.

\begin{cor}
Let $X$ be a smooth quasi-projective surface over a number field $K$, and let $L$ be a foliation on $X$. Suppose that $L$ is closed under $p$-th powers for almost all primes $p$.
Let $P\in X(K)$ be a singular point of $L$. Suppose that $P_\mathbb{C}$ is a non-degenerate reduced singularity of $L_\mathbb{C}$ for some embedding $K \subset \mathbb{C}$. 
Let $\wh{X}_P$ be the completion of $X$ at $P$, and let $\wh{V} \subset \wh{X}_P$ be a formal separatrix through $P$. Suppose in addition that there exists an embedding $K\subset\mathbb{C}$
such that $\wh{V}_\mathbb{C}$ satisfies the Liouville property. Then $\wh{V}$ is algebraic. 
\end{cor}

\subsection{Structure of the paper} Section 2 gathers notation, known results and global conventions that will be used throughout the paper. In section 3 we provide technical tools for the proof of the main results. More precisely, we study formal power series solutions of certain $p$-adic nonlinear differential equation using a Newton iteration procedure.
Section 4 is devoted to the proof of Theorem \ref{thm_intro}.

\section{Notation, conventions and used facts}\label{section:notation}

\subsection{Global convention} Throughout the paper a \textit{variety} is a reduced and irreducible scheme separated and of finite type over a field.

\subsection{Foliations}

Let $K$ be a number field, and let $R$ be its ring of integers. Let $X$ be a smooth quasi-projective variety over $K$. A \textit{foliation} on $X$ is a line bundle $L \subseteq T_X$ such that the quotient $T_X/L$ is torsion free. 

Let $U \subseteq X$ be the open subset where $L|_U$ is a subbundle of $T_U$. We say that $L$ is \textit{singular} at $P \in X$ if $P \in X \setminus U$.

Let $Y \subseteq X$ be a closed subvariety, and let $D$ be a derivation on $X$. 
Say that $Y$ is \textit{invariant under} $D$ if $D(\sI_Y)\subseteq \sI_Y$.

Say that $Y$ is \textit{invariant under} $L$ if for any local section $D$ of $L$ over some open subset $U$ of $X$, $D(\sI_{Y\cap U})\subseteq \sI_{Y\cap U}$.
To prove that $Y$ is invariant under $L$ it is enough to show that  $Y \cap U$ of $Y$ is invariant under $L|_U$ for some open set $U\subseteq X$ such that $Y\cap U$ is dense in $Y$.
If $X$ and $Y$ are smooth and $L \subseteq T_X$ is a subbundle, then $Y$ is invariant under $L$ if and only if $L|_Y \subseteq T_Y \subseteq T_X|_Y$.

\medskip

Let $N \ge 1$ be a sufficiently divisible integer, and let $\sX$ (resp. $\sL$) be a smooth model of $X$ over $R[1/N]$ (resp. a line bundle on $\sX$ contained in $T_{\sX/S}$ such that $\sL\otimes_R K$ coincides with $L$ and the quotient $T_{\sX/S}/\sL$ is torsion free), where $S:=\textup{Spec}\,R[1/N]$. Let $\mathfrak{p}$ be a maximal ideal of $R$ with $\mathfrak{p}\,\nmid\, N$, and let $k_\mathfrak{p}:=R/\mathfrak{p}$ denote the residue field at $\mathfrak{p}$. Let $p$ denote the characteristic of $k_\mathfrak{p}$. The sheaf of derivations $\textup{Der}_{k_\mathfrak{p}}(\sO_{\sX_\mathfrak{p}})\cong T_{\sX_\mathfrak{p}}$ is endowed with the $p$-th power operation, which maps any local $k_\mathfrak{p}$-derivation of $\sO_{\sX_\mathfrak{p}}$ to its $p$-th iterate. 

We say that $L$ \textit{is closed under $p$-th powers for almost all primes $p$} if there exists $N \mid N'$ such that, for any $\mathfrak{p}\,\nmid\, N'$, $\sL|_{\sX_\mathfrak{p}}\subseteq T_{\sX_\mathfrak{p}}$ is closed under $p$-th powers. This condition is independent of the choices of $\sX$ and $\sL$.

\subsection{$A$-analyticity of formal smooth schemes}\label{subsection:A-analytic curves}
We briefly recall a number of definitions and facts concerning $A$-analytic formal curves from \cite{bost} (see also \cite{bost_acl}). We refer to \textit{loc. cit.} for further
explanations concerning these notions. 

\medskip

Let $K$ be field equipped with some complete ultrametric absolute value $|\cdot|$ and assume that its valuation ring $R$ is a discrete valuation ring.

Let $N$ be a positive integer, and let $r$ be a positive real number $r$. For any formal power series $f=\sum_{I\in \mathbb{N}^N} a_I X^I \in \mathbb{C}_p[[X_1,\ldots,X_N]]$, we define $||f||_r$ by the formula
$$||f||_r=\sup_I\,|a_I|r^{|I|} \in \mathbb{R}_{\ge 0}\cup{+\infty}.$$
The power series $f$ such that $||f||_r<+\infty$ are those that are convergent and bounded on the open $N$-ball of radius $r$ in $K^{\,N}$. 

\begin{fact}
Let $r>0$ be a real number, and let $f \in \mathbb{C}_p[[X_1,\ldots,X_N]]$ and $g \in \mathbb{C}_p[[X_1,\ldots,X_N]]$ be formal power series with $||f||_r<+\infty$ and $||g||_r <+\infty$. Then 
$||fg||_r<+\infty$, and $||fg||_r \le ||f||_r||g||_r$. 
\end{fact}

\medskip

For any positive real number $r$, we denote by $G_{\textup{an},r} < \textup{Aut}\big({\wh{\mathbb{A}^N_K}}_0\big)$ the subgroup consisting of all $N$-tuples $f=(f_1,\ldots,f_N)$ such that 
$f(0)=0$, $Df(0)\in \textup{GL}_2(R)$, and $||f_i||_r \le r$ for each $i$.
This subgroup may be identified with the group of all analytic automorphisms, preserving
the origin, of the open $N$-dimensional ball of radius $r$. Set also $G_\textup{an}:=\cup_{r>0} G_{\textup{an},r}$.

\medskip

Notice that a smooth formal subscheme $\wh{V}$ of dimension $d$ of ${\wh{\mathbb{A}^N_K}}_0$ is $K$-analytic if and only if there exists $f \in G_\textup{an}$ such that $f^{-1}\wh{V}$ is the formal subscheme ${\wh{\mathbb{A}^d_K}}_0 \times \{0\}$ of ${\wh{\mathbb{A}^N_K}}_0$.

\medskip

Let $\sX$ be a quasi-projective $R$-scheme, and $X=:\sX\otimes_R K$ its generic fiber. Let $\sP\in \sX(R)$, and let $P:=\sP\otimes_R K$. Let $\wh{V}$ be a smooth formal subscheme of the completion $\wh{X}_P$ of $X$ at $P$.
There is a unique way to attach a number $S_\sX(\wh{V})\in [0,1]$ such that the following holds (see \cite{bost}). 
\begin{enumerate}
\item We have $S_\sX(\wh{V})>0$ if and only if $\wh{V}$ is $K$-analytic.
\item If $(\sX,\sP)=(\mathbb{A}^N_R,0)$ and $\wh{V}$ is $K$-analytic, then $S_\sX(\wh{V})$ is supremum of
the set of real numbers $0<r\le 1$ for which there exists $f \in G_\textup{an}$ such that $f^{-1}\wh{V}$ is the formal subscheme ${\wh{\mathbb{A}^d_K}}_0 \times \{0\}$ of ${\wh{\mathbb{A}^N_K}}_0$.
\item If $\sX \to \sX'$ is an immersion, then $S_\sX(\wh{V})=S_{\sX'}(\wh{V})$.
\item For any two triples $(\sX,\sP,\wh{V})$ and $(\sX',\sP',\wh{V}')$ as above, if there exists 
an $R$-morphism $f \colon \sX \to \sX'$ mapping $\sP$ to $\sP'$, \'etale along $\sP$, and inducing an isomorphism $\wh{V}\cong \wh{V}'$, then $S_\sX(\wh{V})=S_{\sX'}(\wh{V}')$.
\end{enumerate}
We will refer to $S_\sX(\wh{V})$ as the \textit{size of $\wh{V}$ with respect to the model $\sX$ of $X$.}

\medskip

We will need the following easy observations.

\begin{lemma}[\cite{bost_acl}]\label{lemma:size_graph_1}
Let $\phi=\sum_{m\ge 1}c_m X^m \in K[[X]]$ be formal power series such that $c_1 \in R$, and let $\wh{V}$ be its graph in ${\wh{\mathbb{A}^2_K}}_0$. Set $\lambda:=\inf_{m\ge 1}-\frac{\log |c_{m+1}|}{m}\in \mathbb{R}\cup\{-\infty\}$, and let $\rho$ be the radius of convergence of $\phi$. Then the following holds.
\begin{enumerate}
\item The size $S_{\mathbb{A}^2_R}(\wh{V})$ of $\wh{V}$ with respect to $(\mathbb{A}^2_R,0)$ satisfies $S_{\mathbb{A}^2_R}(\wh{V}) \le \rho$.
\item If $\rho>0$, then $S_{\mathbb{A}^2_R}(\wh{V}) \ge \min(1,\exp \lambda)$. In particular, $\wh{V}$ is $K$-analytic. Moreover, if $\phi\in R[[X]]$, then $S_{\mathbb{A}^2_R}(\wh{V})=1$.
\item If $\rho>0$ and $\phi'(0)$ is a unit in $R$, then 
$S_{\mathbb{A}^2_R}(\wh{V}) = \min(1,\exp \lambda)$.
\end{enumerate}
\end{lemma}

\begin{proof}
Items (1) and (3) are shown in \cite[Proposition 3.5]{bost_acl}. To prove Item (2), let $$f(X_1,X_2):=(X_1+X_2,\phi(X_1))\in \textup{Aut}\big({\wh{\mathbb{A}^2_K}}_0\big).$$ Then $f^{-1}\wh{V}$ is the formal subscheme ${\wh{\mathbb{A}^1_K}}_0 \times \{0\}$ of ${\wh{\mathbb{A}^2_K}}_0$. If $\rho >0$, then $\lambda \in \mathbb{R}$. Set $r:=\min(1,\exp \lambda)\in [0,1]$. Then $||f||_r \le r$, and hence $f \in G_\textup{an}$. This finishes the proof the lemma.
\end{proof}

\begin{lemma}\label{lemma:size_graph_2}
Let $\phi=\sum_{m\ge 2}c_m X^m \in K[[X]]$ be formal power series such that  $c_2\in R$, and set $\phi:=\phi(X)/X=\sum_{m\ge 1}c_{m+1} X^m$. 
Let $\wh{V}$ $($resp. $\wh{W})$ be the graph of $\phi$ $($resp. $\psi)$ in ${\wh{\mathbb{A}^2_K}}_0$. Then $S_{\mathbb{A}^2_R}(\wh{V}) \ge S_{\mathbb{A}^2_R}(\wh{W})$.
\end{lemma}

\begin{proof}
Let be a positive real number such that $r<S_{\mathbb{A}^2_R}(\wh{W})$. By assumption, there exist formal power series $f_1$ and $f_2$ in $K[[X_1,X_2]]$ such that $f=(f_1,f_2) \in G_{\textup{an},r}$ and 
$f^{-1}\wh{W}=\wh{A^1_K}\times\{0\}$. This last condition is actually equivalent to the identity
$$f_2(T,0)=\psi(f_1(T,0))$$
in $K[[T]]$.
Let us write $f_1(T,0)=\sum_{m\ge 1}a_m T^m$, and $f_2(T,0)=\sum_{m\ge 1}b_m T^m$.
We have $b_1=c_2a_1$. Since $f \in G_{\textup{an},r}$, we must have $Df(0)\in \textup{GL}_2(R)$.  
This immediately implies that $a_1$ is a unit in $R$. 

Let us set 
$$g_1(X_1,X_2)=f_1(X_1,X_2),\quad g_2(X_1,X_2)=Y+f_1(X_1,X_2)f_2(X_1,X_2),\quad \textup{and} \quad g=(g_1,g_2).$$ 
Then $||g_1||_r=||f_1||_r \le r$ and $||g_2||_r \le r$ since  
$||f_1f_2||_r \le ||f_1||_r||f_2||_r \le r^2\le r$. Moreover, 
$$
Dg(0)=\begin{pmatrix}
a_1 & \partial_{X_2} f (0) \\
0 & 1 
\end{pmatrix} 
\in \textup{GL}_2(R),
$$
and hence $g \in G_{\textup{an},r}$.
Finally, notice that $g^{-1}\wh{V}=\wh{A^1_K}\times\{0\}$ since $f_2(T,0)f_1(T,0)=\phi(f_1(T,0))$ in $K[[T]]$. This shows that $S_{\mathbb{A}^2_R}(\wh{V}) \ge S_{\mathbb{A}^2_R}(\wh{W})$, finishing the proof of the lemma.
\end{proof}

\begin{rem}
In the setup of Lemma \ref{lemma:size_graph_2}, note that the (formal) curve $T \mapsto (T,\psi(T))$ is the proper transform of the (formal) curve $T \mapsto (T,\phi(T))$ in the blow-up of $\mathbb{A}^1_K$ at $0$.
\end{rem}

Let $K$ be a number field, and let $R$ be its ring of integers. Let $\mathfrak{p}$ be a maximal ideal of $R$, let $|\cdot|_\mathfrak{p}$ be the $\mathfrak{p}$-adic absolute value, normalized by the condition $|\varpi_\mathfrak{p}|_\mathfrak{p}=\frac{1}{\sharp (R/\mathfrak{p})}$ for any 
uniformizing element $\varpi_\mathfrak{p}$ at $\mathfrak{p}$. Let 
$K_\mathfrak{p}$ and $R_\mathfrak{p}$ be the $\mathfrak{p}$-adic completions of $K$ and $R$, and 
$k_\mathfrak{p}:=R_\mathfrak{p}/(\varpi_\mathfrak{p})$ the residue field at $\mathfrak{p}$. Let $p$ denote the characteristic of $k_\mathfrak{p}$. Let $X$ be a quasi-projective algebraic variety over $K$, and let $P \in X(K)$. Let $\wh{V}$ be a smooth formal subscheme (defined over $K$) of the formal completion $\wh{X}_P$ of $X$ at $P$.
Let $N \ge 1$ be a sufficiently divisible integer, and let $\sX$ be a quasi-projective model of $X$ over $R[1/N]$, such that $P$ extends to a point $\sP \in \sX(R[1/N])$. The smooth formal scheme $\wh{V}$ is said to be \textit{$A$-analytic} (see \cite[Definition 3.7]{bost_acl}) if 
\begin{enumerate}
\item for any place $v$ of $K$, the formal scheme $\wh{V}_{K_v}$ is $K_v$-analytic, where $K_v$ denotes the completion of $K$ with respect to $v$, and
\item we have $$\sum_{\mathfrak{p}\,\nmid\, N} \log \frac{1}{S_{\sX_{R_\mathfrak{p}}}(\wh{V}_{K_\mathfrak{p}})}<+\infty.$$
\end{enumerate}
This condition is independent of the choices of $\sX$, $\sP$, and $N$.

\section{Power series solutions of certain $p$-adic differential equations}

Let $p$ be a prime number. Let $\mathbb{Q}_p$ (resp. $\mathbb{C}_p$) be the field of $p$-adic rational numbers (resp. complex numbers). We denote by $|\cdot|$ the ultrametric absolute value on $\mathbb{C}_p$ normalized by $|p|=p^{-1}$. We denote by $v$ the $p$-adic valuation on $\mathbb{C}_p$ normalized by $v(p)=1$.

\medskip

In this section, we study formal power series solutions of certain $p$-adic nonlinear differential equations at a regular singular point. The following is the main result of this section.

\begin{prop}\label{prop:size_solution_ode}
Let $p$ be an odd prime integer, let $1 \le s \le p-1$ and $1 \le t \le p-1$ be relatively prime integers, and set $\alpha:=\frac{s}{t}$.
Let $a$, $b$, and $c_m$ for any integer $m\ge 2$ be power series in $\mathbb{C}_p[[X]]$
such that $||a||_r \le \frac{r}{p}$, $||b||_r \le \frac{r}{p}$, and $||c_m||_r \le \frac{1}{p}$ for some real number $r \in ]0,1]$. Suppose in addition that $a(0)=a'(0)=0$ and that $b(0)=0$.
Then the following holds.
\begin{enumerate}
\item There exists a unique formal power series $y$ in $\mathbb{C}_p[[X]]$ with $y(0)=y'(0)=0$ solution of the differential equation 
\begin{equation}\label{eq:edo_-1}
xy'+\alpha y = a+by+\sum_{m\ge 2}c_my^m.
\end{equation}
\item There exists a constant $C>0$ such that, letting 
$R:=r \exp\Big(- Ct\frac{(\log p)^2}{p^2}\Big)$, we have $||y||_{R} \le R$.
\end{enumerate}
\end{prop}

\medskip

We will need the following well-known facts.

\begin{fact}
Let $f=\sum_{m\ge 0}a_mX^m\in\mathbb{C}_p[[X]]$ and $g=\sum_{m\ge 0}b_mX^m\in\mathbb{C}_p[[X]]$. Suppose that there exists a positive real number $r$ such that both $|a_m|r^m$ and $|b_m|r^m$ goes to $0$ as $m$ goes to $+\infty$. Then $||fg||_r=||f||_r||g||_r$ (see \cite[Proposition 2 of Section 6.1.4]{robert}).
\end{fact}

\begin{fact}
Let $f\in\mathbb{C}_p[[X]]$. If $f(0)=0$ and $||f||_r \le Cr$ for some positive real numbers $r$ and $C$, then $||f||_{r_1} \le Cr_1$ for $0<r_1 \le r$. 
\end{fact}

\begin{fact}
Let $f=\sum_{m\ge 0}a_mX^m\in\mathbb{C}_p[[X]]$. Suppose that there exists a positive real number $r$ such that $|a_m|r^m$ goes to $0$ as $m$ goes to $+\infty$. Then
$||f||_r=||f_r||_1=\sup_{|x|\le 1}|f_r(x)|=\sup_{|x|\le r}|f(x)|$,
where $f_r(X):=f(rX)=\sum_{m\ge 0}a_mr^mX^m$ (see \cite[Proposition 1 of Section 6.1.4]{robert}).
\end{fact}

Before we give the proof of Proposition \ref{prop:size_solution_ode}, we need the following auxiliary statements.

\begin{lemma}\label{lemma:change_of_variable}
Let $k$ be a positive integer, and let $b=\sum_{m\ge 2^k}b_mX^m \in \mathbb{C}_p[[X]]$ be a power series such that 
$||b||_r \le \frac{r}{p}$ for some real number $r \in ]0,1]$. Set $B=\sum_{m\ge 2^k}\frac{b_m}{m}X^m$. 
\begin{enumerate}
\item Let $r_1:=r\exp \big(- \frac{2}{p(p-1)}\log p\big)$.
Then $||B||_{r_1} \le p^{-\frac{2}{p-1}}$.
\item Let $r_1:=r\exp \big(-\frac{k+1}{2^k}\log 2\big)$.
Then $||B||_{r_2} \le p^{-\frac{2}{p-1}}$.
\end{enumerate}
\end{lemma}

\begin{proof}
Let $R$ be a positive real number. Then $||B||_{R} \le p^{-\frac{2}{p-1}}$ if and only if 
$$\log R \le \inf_{m\ge 2^k}\left\{\frac{1}{m}\Big(-\log|b_m|+\log|m|-\frac{2}{p-1}\log p\Big)\right\}.$$
By assumption, we have 
$$\log|b_m|+(m-1)\log r \le -\log p$$
for any integer $m \ge 2^k$, and hence
\begin{align*}
\inf_{m\ge 2^k}\left\{\frac{1}{m}\Big(-\log|b_m|+\log|m|-\frac{2}{p-1}\log p\Big)\right\} &\ge 
\inf_{m\ge 2^k}\left\{\frac{1}{m}\Big((m-1)\log r +\log|m| + \frac{p-3}{p-1}\log p\Big)\right\}\\
& \ge \log r + \inf_{m\ge 2^k}\left\{\frac{1}{m}\Big(\log|m| + \frac{p-3}{p-1}\log p\Big)\right\}
\end{align*}
using the fact that $r \le 1$.

\medskip

Notice that $\log|m| + \frac{p-3}{p-1}\log p\le - \log p + \frac{p-3}{p-1}\log p = -\frac{2}{p-1}\log p <0$ if $|m|<1$, and hence
$$\inf_{m\ge 2^k}\left\{\frac{1}{m}\Big(\log|m| + \frac{p-3}{p-1}\log p\Big)\right\}<0.$$
Then
\begin{align*}
\inf_{m\ge 2^k}\left\{\frac{1}{m}\Big(\log|m| + \frac{p-3}{p-1}\log p\Big)\right\} & \ge
\inf_{m\ge 1}\left\{\frac{1}{m}\Big(\log|m| + \frac{p-3}{p-1}\log p\Big)\right\} \\
& =
\inf_{m\ge 1\text{ such that }|m|<1}\left\{\frac{1}{m}\Big(\log|m| + \frac{p-3}{p-1}\log p\Big)\right\}\\
&=\inf_{t \ge 1\text{ and } u \ge 1 \text{ such that }|u|=1}\left\{\frac{1}{p^{t}u}\Big(-t + \frac{p-3}{p-1}\Big)\log p\right\}\\
&\ge \inf_{t \ge 1}\left\{\frac{1}{p^{t}}\Big(-t + \frac{p-3}{p-1}\Big)\log p\right\}.
\end{align*}
Now observe that $x \mapsto \frac{1}{p^{x}}\Big(-x + \frac{p-3}{p-1}\Big)$ is an increasing function on 
$x > \frac{p-3}{p-1}+\frac{1}{\log p}$. On the other hand, we have $\frac{p-3}{p-1}+\frac{1}{\log p} <2$. Thus 
$$\inf_{t \ge 1}\left\{\frac{1}{p^{t}}\Big(-t + \frac{p-3}{p-1}\Big)\log p\right\}=
\min_{t \in \{1,2\}}\left\{\frac{1}{p^{t}}\Big(-t + \frac{p-3}{p-1}\Big)\log p\right\}=
\frac{-2}{p(p-1)}\log p,$$
and hence 
$$\inf_{m\ge 2^k}\left\{\frac{1}{m}\Big(-\log|b_m|+\log|m|-\frac{2}{p-1}\log p\Big) \right\}\ge \log r - \frac{2}{p(p-1)}\log p.$$
This proves (1).

\medskip

We proceed to show (2). We have
\begin{align*}
\inf_{m\ge 2^k}\left\{\frac{1}{m}\Big(\log|m| + \frac{p-3}{p-1}\log p\Big)\right\}  & \ge 
\inf_{m\ge 2^k}\left\{\frac{1}{m}\Big(-\log m + \frac{p-3}{p-1}\log p\Big)\right\} \\
 & \ge \inf_{m\ge 2^k}\left\{\frac{1}{m}\Big(-\log m + \frac{p-3}{p-1}\log p\Big)\right\} \\
 & \ge \inf_{m\ge 2^k}\left\{-\frac{\log m}{m}\right\}-\frac{1}{2^k}\log 2.
\end{align*}

Notice that the function $x \mapsto -\frac{\log x}{x}$ is increasing on $\log x > 1$. Therefore,
if $k \ge 2$, then 
$$\inf_{m\ge 2^k}\left\{\frac{1}{m}\Big(-\log m \Big)\right\}= -\frac{k}{2^k}\log 2.$$
If $k=1$, then 
$$\inf_{m\ge 2^k}\left\{\frac{1}{m}\Big(-\log m \Big)\right\}= \min \left\{-\frac{1}{2}\log 2,-\frac{2}{2^2}\log 2\right\}=-\frac{1}{2}\log 2 \ge -\log 2.$$

This proves (2), completing the proof of the lemma.
\end{proof}

\begin{rem}
In the setup of Lemma \ref{lemma:change_of_variable}, notice that $\exp(B(x))$ is the unique formal power series solution of the differential equation $xy'=by$ such that $y(0)=1$. We will denote $B(x)$ by 
$\int_0^x\frac{b(u)}{u}du$. 
\end{rem}

\begin{lemma}\label{lemma:ode_special}
Let $k$ be a positive integer, and let $a=\sum_{m\ge 2^k}a_mX^m \in \mathbb{C}_p[[X]]$ be a power series such that $||a||_r \le \frac{r}{p}$ for some positive real number $r$.
Let $1 \le s \le p-1$ and $1 \le t \le p-1$ be relatively prime integers, and set 
$\alpha:=\frac{s}{t}$ and
$A=\sum_{m\ge 2^k}\frac{a_m}{m+\alpha}X^m$.  
\begin{enumerate}
\item Let $r_1:=r \exp\big(-\frac{t}{(p-1)^2}\log p\big)$. 
Then $||A||_{r_1} \le r_1$. 
\item Suppose that $2^k \ge \alpha+2$, and let $r_2:=r\exp\big(-\frac{kt}{2^{k-1}}\log 2\big)$. 
Then $||A||_{r_2} \le r_2$.
\end{enumerate}
\end{lemma}

\begin{proof}
Let $R$ be a positive real number. Then $||A||_{R} \le R$ if and only if 
$$\log R \le \inf_{m \ge 2^k}\left\{-\frac{1}{m-1}\log\frac{|a_m|}{|m+\alpha|}\right\}.$$

By assumption, we have $\log|a_m|+(m-1)\log r \le -\log p$
for any $m \ge 2^k$. Notice also that $|s|=|t|=1$. Thus
\begin{align*}
\inf_{m \ge 2^k}\left\{-\frac{1}{m-1}\log\frac{|a_m|}{|m+\alpha|}\right\} &\ge \log r + \inf_{m\ge 2^k}\left\{\frac{1}{m-1}\big(\log p + \log |tm+s|\big)\right\}\\
&\ge \log r + \inf_{m\ge 2^kt+s}\left\{\frac{t}{m-s-t}\big(\log p + \log |m|\big)\right\}\\
&\ge \log r + \inf_{m\ge 2t+s}\left\{\frac{t}{m-s-t}\big(\log p + \log |m|\big)\right\}.
\end{align*}

\medskip

Note that $\log p + \log |m| < 0$ if $|m|<|p|$, and that $m-s-t \ge t \ge 1$. Hence
$$\inf_{m\ge 2t+s}\left\{\frac{t}{m-s-t}\big(\log p + \log |m|\big)\right\} < 0,$$
and thus
\begin{align*}
\inf_{m\ge 2t+s}\left\{\frac{t}{m-s-t}\big(\log p + \log |m|\big)\right\} & =
\inf_{m\ge 2t+s\text{ and }|m|<|p|}\left\{\frac{t}{m-s-t}\big(\log p + \log |m|\big)\right\}\\
&= \inf_{k\ge 2\text{ and } u \ge 1 \text{ such that }|u|= 1}\left\{\frac{t}{p^ku-s-t}\big(1 - k \big)\log p\right\}\\
&\ge \inf_{k\ge 2}\left\{\frac{t}{p^k-2(p-1)}\big(1 - k \big)\log p\right\}\\
&\ge \inf_{k\ge 2}\left\{\frac{t}{p^k\Big(1-\frac{2(p-1)}{p^2}\Big)}\big(1 - k \big)\log p\right\}\\
&=\frac{t\log p}{\Big(1-\frac{2(p-1)}{p^2}\Big)}\inf_{k\ge 2}\frac{1-k}{p^k}.
\end{align*}
On the other hand, the function $x \mapsto \frac{1-x}{p^x}$ is increasing on $x>1+\frac{1}{\log p}$. This immediately implies 
$$\inf_{k\ge 2}\frac{1-k}{p^k}=-\frac{1}{p^2},$$ and hence

\begin{align*}
\frac{t\log p}{\Big(1-\frac{2(p-1)}{p^2}\Big)}\inf_{k\ge 2}\frac{1-k}{p^k}
&=-\frac{t}{p^2\Big(1-\frac{2(p-1)}{p^2}\Big)}\log p\\
&=-\frac{t}{p^2-2p+2}\log p\\
&\ge -\frac{t}{(p-1)^2}\log p.
\end{align*}
This proves (1).

\medskip

We proceed to show (2). Suppose that $2^k \ge \alpha+2$. If moreover $m\ge 2^kt+s$, then 
$$\frac{m}{2}-s-t \ge \frac{1}{2}(2^kt+s)-s-t\ge \frac{1}{2}(s+2t+s)-s-t=0,$$ and hence $m-s-t \ge \frac{m}{2}$. Notice in addition that we must have $k\ge 2$ since $\alpha>0$ by assumption.
Then, we have

\begin{align*}
\inf_{m\ge 2^kt+s}\left\{\frac{t}{m-s-t}\big(\log p + \log |m|\big)\right\} & \ge 
\inf_{m\ge 2^kt+s}\left\{\frac{t}{m-s-t}\big(\log p - \log m\big)\right\}\\
& \ge -\inf_{m\ge 2^kt+s}\left\{\frac{t}{m-s-t}\log m\right\}\\
& \ge -2t \inf_{m\ge 2^kt+s}\frac{\log m}{m}\\
& \ge -2t \inf_{m\ge 2^k}\frac{\log m}{m}\\
& = -\frac{kt}{2^{k-1}}\log 2.
\end{align*}
This completes the proof of the lemma.
\end{proof}

\begin{rem}\label{remark:ode_special}
In the setup of Lemma \ref{lemma:ode_special}, note that $A$ is the unique formal power series solution of the differential equation $$xy'+\alpha y=a$$ with $y(0)=y'(0)=0$.
\end{rem}

We are now in position to prove Proposition \ref{prop:size_solution_ode}.

\begin{proof}[{Proof of Proposition \ref{prop:size_solution_ode}}]
Notice that $m+\alpha \neq 0$ for any integer $m \ge 2$. Item (1) then follows easily.

\medskip

The proof of Item (2) is subdivided into a number of steps.

\medskip

\noindent\textit{Step 1.} Let $m\ge 2$ be an integer. Observe that 
$v(m+\alpha)=v(tm+s) \le \frac{\log(tm+s)}{\log p}$. Therefore, \cite[Theorem 2]{sibuya_sperber_1} applies to show that the formal solution $y$ is convergent.

Set $$c(X_1,X_2):=\sum_{m\ge 2}c_m(X_1)X_2^m \in \mathbb{C}_p[[X_1,X_2]],$$
and consider the formal power series 
$$B(x):=\int_0^x\frac{b(u)}{u}du\quad\text{and}\quad z(x):=y(x)\exp(-B(x)).$$
Note that $z(0)=z'(0)=0$. Then $y$ is solution of Equation \eqref{eq:edo_-1} if and only if $z$ is solution of 
\begin{equation*}
xz'(x)+\alpha z(x) =\exp(-B(x))\big(a(x)+c(x,z(x)\exp(B(x)))\big).
\end{equation*}
By Lemma \ref{lemma:change_of_variable} (1), $B(x)$ converges if
$|x| \le r_0:=r \exp\big(-\frac{2}{p(p-1)}\log p\big)$, and $|B(x)|\le p^{-\frac{2}{p-1}} <p^{-\frac{1}{p-1}}$. In particular, $\exp(\pm B(x))$ is well-defined. 
Set
$$a_{0}(x):=a(x)\exp(-B(x))) \quad \text{and} \quad c_{0,m}(x):= c_{m}(x)
\exp((m-1)B(x))$$
for all $m\ge 2$, so that 
$$\exp(-B(x))c(x,z(x)\exp(B(x)))=\sum_{m\ge 2}c_{0,m}(x)z(x)^m.$$

Let $r_1:=r \exp\big(-\frac{3}{p(p-1)}\log p\big) <r_0$.
Then \cite[Theorem of Section 6.1.5]{robert} applies to show that $\exp(\pm B)(x)$ converges and that $\exp(\pm B)(x)=\exp(\pm (B(x))$ if $|x|\le r_1$. 
In addition, we have $||\exp(\pm B)||_{r_1}=\sup_{|x|\le r_1}|\exp(\pm B)(x)|=\sup_{|x|\le r_1}|\exp(\pm (B(x))|$. On the other hand, 
$\sup_{|x|\le r_1}|\exp(\pm(B(x))|\le 1$ since $|m!|\ge p^{-\frac{m}{p-1}}$ for every integer $m\ge 1$. It follows that $$||\exp(\pm B)||_{r_1}=1$$ since $B(0)=0$. As a consequence, we have
$$||a_0||_{r_1}=||a||_{r_1}||\exp(-B)||_{r_1}=||a||_{r_1}\le \frac{r_1}{p},$$ 
$$||c_{0,m}||_{r_1}=||c_{m}||_{r_1} \le ||c_{m}||_{r} \le \frac{1}{p},$$ and 
$$||y||_{R}=||z||_{R}$$ for any real number $0<R\le r_1$.

\medskip

Replacing $r$ by $r_1$, if necessary, we may therefore assume without loss of generality that $b=0$, so that Equation \eqref{eq:edo_-1} reads
\begin{equation}\label{eq:edo_-0}
xy'+\alpha y = a+\sum_{m\ge 2}c_my^m.
\end{equation}

\medskip

\noindent\textit{Step 2.} The proof of \cite[Theorem 2]{sibuya_sperber_1} then goes as follows. Let $z_0 \in \mathbb{C}_p[[X]]$ be the unique formal power series solution of 
$$
xz_0'(x)+\alpha z_0(x)=a(x)
$$ 
with $z_0(x)=O(x^2)$ as $x$ goes to $0$. Set $b_0=0$. Observe that $a_1(x):=c(x,z_0(x))=O\big(x^{2^2}\big)$ as $x$ goes to $0$.
Set $b_1(x):=\partial_{x_2}c(x,z_0(x))$, and 
let $z_1 \in \mathbb{C}_p[[X]]$ be the unique formal power series solution of 
$$
xz_1'(x)+\alpha z_1(x)=a_1(x)+b_1(x)z_1(x)
$$
with $z_1(x)=O\big(x^{2^2}\big)$ as $x$ goes to $0$. Next, one defines inductively $z_k \in \mathbb{C}_p[[X]]$ for all integer $k \ge 2$ as follows. Set 
$$y_k:=\sum_{i=0}^{k}z_i$$
for $k \ge 0$. If $k \ge 2$, set 
$$a_k(x):=c(x,y_{k-1}(x))-c(x,y_{k-2}(x))-z_{k-1}(x)\partial_{x_2}c(x,y_{k-2}(x))$$
and
$$b_k(x):=\partial_{x_2}c(x,y_{k-1}(x)).$$
For $k\ge 2$, one then proves that there exists a unique formal power series $z_k$ solution of
\begin{equation}\label{eq:edo_k}
xz_k'(x)+\alpha z_k(x)=a_k(x)+b_k(x)z_k(x)
\end{equation}
with 
$z_k(x)=O\big(x^{2^{k+1}}\big)$. 
Finally, one proves that $z_k$ converges for all $k \ge 0$ as well as $y:=\sum_{k\ge 0}z_k$ and that $y$ is the unique solution formal power series solution of Equation \eqref{eq:edo_-0} with $y(x)=O(x^2)$ as $x$ goes to $0$.

\medskip

\noindent\textit{Step 3.} 
Set $B_k(x):=\int_0^x\frac{b_k(u)}{u}du$. Notice that, for any integer $k\ge 1$, the formal power series $z_k$ is solution of Equation $\eqref{eq:edo_k}$ if and only if 
$$w_k(x):=z_k(x)\exp(-B_k(x))$$ is solution of 
$$xw_k'(x)+\alpha w_k(x) =\exp(-B_k(x))a_k(x).$$

\medskip

Let $k_1$ be the smallest positive integer such that $\frac{k_1+1}{2^{k_1}}\le \frac{1}{p^2}$. Then $2^{k_1} \ge p^2+1 \ge p+1 \ge \alpha+2$. Notice that $k_1 \le 5 \log p$. Indeed, let $k$ be any integer such that $k \ge \frac{2}{\log 2}\log p+1$. Then $\frac{k+1}{2^k} \le \frac{1}{2^{k-1}}\le \frac{1}{p^2}$. It follows that 
$k_1 \le \frac{2}{\log 2}\log p+2 \le 3 \log p +2 \le 5 \log p$.

Then, we define inductively a decreasing sequence $(r_k)_{k \ge 0}$ of positive real numbers such that the following holds. For any integer $k \ge 0$,  if $|x|\le r_k$, then $\exp(\pm B_k)(x)$ converges, $\exp(\pm B_k)(x)=\exp(\pm (B_k(x))$, and $z_k(x)$ converges as well. In addition, $$||\exp(\pm B_k)||_{r_k}=1,$$ and $$||z_{k}||_{r_{k}} \le r_{k}.$$ Finally, we will show that 
the limit $r_\infty$ of the sequence $(r_k)_{k \ge 0}$ satisfies 
$$\log r - \log r_\infty \le Ct\frac{(\log p)^2}{p^2}$$
for some constant $C>0$. In particular, $r_\infty > 0$.

Set $r_0:= r \exp\big(-\frac{t}{(p-1)^2}\log p\big)$. By Lemma \ref{lemma:ode_special} (1), $||z_0||_{r_0} \le r_0$. Moreover, $B_0=0$ since $b_0=0$.

Let now $k$ be a positive integer. Suppose $r_{k-1} < r_{k-2} < \cdots < r_0$ have already been defined. For any integer $0 \le j \le k-1$, we have 
$$||y_j||_{r_{k-1}} \le \max_{0\le i \le j}||z_i||_{r_{k-1}} \le r_{k-1},$$
and $y_j(x)$ converges if $|x|\le r_{k-1}$. Moreover, 
\begin{align*}
||b_{k}||_{r_{k-1}}  & = ||\partial_{x_2}c\big(x,y_{k-1}(x)\big) ||_{r_{k-1}} \\
& = ||\sum_{m\ge 2} mc_m y_{k-1}^{m-1}||_{r_{k-1}}\\
& \le \max_{m\ge 2} \left\{ |m|||c_m||_{r_{k-1}} ||y_{k-1}||^{m-1}_{r_{k-1}} \right\}\\
& \le \max_{m\ge 2} \left\{ ||c_m||_{r} r_{k-1}^{m-1} \right\}\\
& \le \frac{r_{k-1}}{p}.
\end{align*}
Similarly, we have
$$||a_{k}||_{r_{k-1}} 
= ||c(x,y_{k-1}(x))-c(x,y_{k-2}(x))-z_{k-1}(x)\partial_{x_2}c(x,y_{k-2}(x))||_{r_{k-1}}\le \frac{r_{k-1}}{p}.$$

Suppose first $k<k_1$. Set  
$$r_k:=r_{k-1} \exp\Big(-\frac{2t}{(p-1)^2}\log p\Big) \le r_{k-1} \exp\Big( -\frac{2}{p(p-1)}\log p\Big) < r_{k-1}.$$
By Lemma \ref{lemma:change_of_variable} (1) applied to $b_k$ and 
Lemma \ref{lemma:ode_special} (1) applied to $\exp(-B_k(x))a_k(x)$ and arguing as in Step 1, we see that $r_k$ satisfies all the conditions listed above.

Suppose now that $k \ge k_1$, and set  
$$r_k:=r_{k-1}\exp\Big(-\frac{(k+1)t}{2^{k-1}}\log 2\Big)\le r_{k-1}\exp\Big( -\frac{k+1}{2^k}\log 2\Big)<r_{k-1}.$$ 
Applying Lemma \ref{lemma:change_of_variable} (2) to $b_k$ and 
Lemma \ref{lemma:ode_special} (2) to $\exp(-B_k(x))a_k(x)$ and arguing again as in Step 1, we see that $r_k$ also satisfies all the conditions listed above.

Finally, we have
$$\sum_{1 \le k \le k_1-1} (\log r_{k-1} -\log r_{k})=\sum_{1 \le k \le k_1-1}\frac{2t}{(p-1)^2}\log p = \frac{2(k_1-1)t}{(p-1)^2}\log p \le 10t\frac{(\log p)^2}{(p-1)^2},$$
and
\begin{align*}
\sum_{k\ge k_1} (\log r_{k-1} -\log r_{k}) & = \sum_{i \ge k_1}\frac{(k+1)t}{2^{k-1}}\log 2 \\
& \le t\log 2 \int_{k_1}^{+\infty}\frac{x+1}{2^{x-1}}dx\\
& = \frac{t}{2^{k_1-1}}\left(\frac{1}{\log 2}+k_1+1\right)\\
& \le 4t \frac{k_1+1}{2^{k_1}}\\
& \le \frac{4t}{p^2}.
\end{align*}
Therefore, we have
$$\log r - \log r_\infty = \sum_{k\ge 1} (\log r_{k-1} -\log r_{k})\le 10t\frac{(\log p)^2}{(p-1)^2} + \frac{4t}{p^2}
\le 14t\frac{(\log p)^2}{(p-1)^2}.$$
Set $C:= 14$, and $R:=r \exp\big(-tC\frac{(\log p)^2}{p^2}\big)$. Then $R \le r_k$ for any $k \ge 0$, and $||y||_R \le R$ since $||y_k||_R\le R$ for each $k\ge 0$.
This completes the proof of the proposition.
\end{proof}

\section{Proof of Theorem \ref{thm_intro}}

In this section we prove our main result. Note that Theorem \ref{thm_intro} is an immediate consequence of Theorem \ref{thm:A-analyticity} below.

\medskip

Let $X$ a smooth complex quasi-projective surface, and let $L$ be a foliation on $X$. Let $P$ be a singular point of $L$, and let $D$ be a vector field on some open subset $U \ni P$ such that $L|_U=\sO_U D$. Let $\alpha_1$ and $\alpha_2$ be the eigenvalues of the linear part of $D$ at $P$. Recall that $P$ is a \textit{reduced singularity} of $L$ if at least one of the $\alpha_i$'s is non-zero, say $\alpha_2$, and  $\frac{\alpha_1}{\alpha_2}$ is not a positive rational number. A reduced singularity $P$ is called \textit{non-degenerate} if both $\alpha_1$ and $\alpha_2$ are non-zero. Then $\big\{\alpha,\frac{1}{\alpha}\big\}$ does not depend on the choice of $D$, where $\alpha:=\frac{\alpha_1}{\alpha_2}$. Finally, recall from \cite[Appendice II]{mattei_moussu}, that if $P$ is a non-degenerate reduced singularity, then there are exactly two analytic curves in $X$ passing through $P$ and invariant under $L$. They are smooth and intersect transversely at $P$.

\begin{thm}\label{thm:A-analyticity}
Let $X$ be a smooth quasi-projective surface over a number field $K$, and let $L$ be a foliation on $X$. Suppose that $L$ is closed under $p$-th powers for almost all primes $p$.
Let $P\in X(K)$ be a singular point of $L$. Suppose that $P_\mathbb{C}$ is a reduced singularity of $L_\mathbb{C}$ for some embedding $K \subset \mathbb{C}$. Then the following holds.
\begin{enumerate}
\item The foliation $L$ has a non-degenerate singularity at $P$, and $-\alpha_{P_\mathbb{C}}(L_\mathbb{C}) \in \mathbb{Q}_{> 0}$. In particular, $P_\mathbb{C}$ is a reduced singularity of $L_\mathbb{C}$ for any embedding $K \subset \mathbb{C}$.
\item There exist two $\wh{L}$-invariant smooth formal subschemes $\wh{V}$ and $\wh{W}$ of the completion $\wh{X}_P$ of $X$ at $P$ defined over $K$, where $\wh{L}$ denotes the $($formal$)$ foliation induced by $L$ on $\wh{X}_P$. In particular, $\wh{V}_\mathbb{C}$ and $\wh{W}_\mathbb{C}$ are the formal completion at $P$ of the two separatrices of $L_\mathbb{C}$ through $P_\mathbb{C}$ for any embedding $K \subset \mathbb{C}$.
\item  The formal curves $\wh{V}$ and $\wh{W}$ are $A$-analytic.
\end{enumerate}
\end{thm}

\begin{proof}
Replacing $X$ by a Zariski open neighborhood of $P$ in $X$, if necessary,
we may assume without loss of generality that $X$ is affine, and that $L=\sO_X D$ for some vector field $D\in T_X \cong\textup{Der}_{k}(\sO_X)$. Let $\wb{K}$ be an algebraic closure of $K$, and let 
$\alpha_1 \in \wb{K}$ and $\alpha_2 \in \wb{K}$ be the eigenvalues of the linear part of $D$ at $P$. Suppose $\alpha_2 \neq 0$.
By \cite[Proposition II.1.3]{mcquillan08}, $D$ is formally linearisable at $P$ and
$\frac{\alpha_1}{\alpha_2} \in \mathbb{Q}$. This immediately implies that $\alpha_1\neq 0$ since $D$ has isolated zeroes by assumption. In addition, we must have $-\alpha_{P_\mathbb{C}}(L_\mathbb{C}) = -\frac{\alpha_1}{\alpha_2}\in \mathbb{Q}_{> 0}$ for any embedding $K \subset \mathbb{C}$, proving (1).

Let us write $\alpha=-\frac{s}{t}$, where $s$ and $t$ are relatively prime positive integers, and let
$K \subset \mathbb{C}$ be an embedding. Let $x_1$ and $x_2$ be
regular functions on $X$ such that the induced map 
$X \to \mathbb{A}_K^2$ is \'etale at $P$, and maps $P$ to $0$. Shrinking $X$ further, we may assume that $X \to \mathbb{A}_K^2$ is \'etale.
Let $\wh{D}$ denote the $K$-derivation of $\wh{\sO}_{X,P}\cong K[[X_1,X_2]]$ induced by
$D$. We may assume without loss of generality that 
$$\wh{D}=(X_1+f_1(X_1,X_2))\partial_{X_1}+(\lambda X_2+f_2(X_1,X_2))\partial_{X_2},$$  
where $f_1$ and $f_2$ are formal power series with coefficients in $K$ vanishing to order at least $2$ at $P$. Let $\phi_1$ and $\phi_2$ be formal power series with coefficients in $K$ vanishing to order at least $2$ at $0$. Then the formal (smooth) curves $$T\mapsto(\phi_1(T),T)\quad\text{and}\quad T\mapsto (T,\phi_2(T))$$ are invariant under $\wh{D}$ if and only if 
$$\wh{D}(X_1-\phi_1(X_2)) \in (X_1-\phi_1(X_2))\quad\text{and}\quad \wh{D}(X_2-\phi_2(X_1)) \in (X_2-\phi_2(X_1))$$
if and only if
\begin{equation}\label{eq:separatrix_1}
\phi_1(T)+f_1(\phi_1(T),T)-\phi'_1(T)\big(\lambda T+f_2(\phi_1(T),T)\big)=0
\end{equation}
and
\begin{equation}\label{eq:separatrix_2}
\lambda\phi_2(T)+f_2(T,\phi_2(T))-\phi'_2(T)\big(T+f_1(T,\phi_2(T))\big)=0.
\end{equation}
One readily checks that there is a unique formal power series $\phi_1\in K[[T]]$ (resp. $\phi_2\in K[[T]]$)
vanishing to order at least $2$ at $0$ solution of Equation \eqref{eq:separatrix_1}
(resp. Equation \eqref{eq:separatrix_2}) since $-\alpha\in\mathbb{Q}_{>0}$. This proves (2).

\medskip

We will denote by $\wh{V}$ (resp. $\wh{W}$) the formal curve $T\mapsto(\phi_1(T),T)$ (resp. $T\mapsto (T,\phi_2(T))$).

\medskip

The proof of Item (3) is subdivided into a number of steps.

\medskip

\noindent\textit{Step 1.} Let $v$ be a place of $K$, and let $K_v$ be the completion of $K$ at $v$. Notice first that $\wh{V}_{K_v}$ and $\wh{W}_{K_v}$ are $K_v$-analytic by 
\cite[Theorem 2]{sibuya_sperber_1} if $v$ is a finite place and \cite[Appendice II]{mattei_moussu} if $v$ is archimedean.

\medskip

\noindent\textit{Step 2.} Suppose that the (algebraic) curve $\{x_1\}=0$ is $\sL$-invariant. 
Shrinking $X$ further, if necessary, we may therefore assume that 
$$D=x_1\partial_{x_1}+f(x_1,x_2)\partial_{x_2},$$
where $f$ is a regular function on $X$.

Let $R$ be the ring of integers of $K$. If $N$ denotes a sufficiently divisible positive integer, there exists a model $\sX$ of $X$, smooth and quasi-projective over $R[1/N]$, such that $P$ extends to a point $\sP\in\sX(R[1/N])$. We may also assume that $X \to \mathbb{A}^2_K$ extends to an \'etale morphism $\sX \to \mathbb{A}^2_{R[1/N]}$, and that $D$ extends to a vector field 
$\sD \in H^0(\sX,T_{\sX/\textup{Spec}\,R[1/N]})$ with isolated zeroes. Let $\mathfrak{p}$ be a maximal ideal of $R$, and let $|\cdot|_\mathfrak{p}$ be the $\mathfrak{p}$-adic absolute value, normalized by the condition $|\varpi_\mathfrak{p}|_\mathfrak{p}=\frac{1}{\sharp (R/\mathfrak{p})}$ for any 
uniformizing element $\varpi_\mathfrak{p}$ at $\mathfrak{p}$. Let 
$K_\mathfrak{p}$ and $R_\mathfrak{p}$ be the $\mathfrak{p}$-adic completions of $K$ and $R$, and 
$k_\mathfrak{p}:=R_\mathfrak{p}/(\varpi_\mathfrak{p})$ the residue field at $\mathfrak{p}$. Let finally $p$ denote the characteristic of $k_\mathfrak{p}$.

Suppose that $\mathfrak{p}\nmid N$, and that $K_\mathfrak{p}$ is absolutely unramified. Let $\wh{\sD}_\mathfrak{p}$ denote the $K_\mathfrak{p}$-derivation of $\wh{\sO}_{\sX_{K_\mathfrak{p}},\sP_{K_\mathfrak{p}}}\cong K_\mathfrak{p}[[X_1,X_2]]$ induced by
$\sD_{K_\mathfrak{p}}=D_{K_\mathfrak{p}}$. Then 
$$\wh{\sD}_\mathfrak{p}=X_1\partial_{X_1}+f_\mathfrak{p}(X_1,X_2)\partial_{X_2}$$
where $f_\mathfrak{p}\in R_\mathfrak{p}[[X_1,X_2]]$. Let also 
$\wb{\sD}_\mathfrak{p}$ be the $k_\mathfrak{p}$-derivation of $k_\mathfrak{p}[[X_1,X_2]]$ induced by $\wh{\sD}_\mathfrak{p}$. By assumption, $\wb{\sD}_\mathfrak{p}$ is p-closed. This immediately implies $$\wb{\sD}_\mathfrak{p}^{\, p}=\wb{\sD}_\mathfrak{p}$$
since $\wb{\sD}_\mathfrak{p}^{\, p}(X_1)=\wb{\sD}_\mathfrak{p}(X_1)=X_1$. By \cite[Lemma 6.4]{aramova_avramov} applied to $\wb{\sD}_\mathfrak{p}$, there exists a formal power series 
$\wb{Y}_2$ in $k_\mathfrak{p}[[X_1,X_2]]$ such that 
$\wb{\sD}_\mathfrak{p}(\wb{Y}_2)=\alpha \wb{Y}_2$.
Notice that $\wb{Y}_2$ may a priori depend on our choice of $\mathfrak{p}$. 
Let $Y_2 \in R_\mathfrak{p}[[X_1,X_2]]$ be any formal power series whose reduction modulo $\varpi_\mathfrak{p}$ is $\wb{Y}_2$. 
By construction, we have 
$\wh{\sD}_\mathfrak{p}(Y_2)=\lambda Y_2$ modulo $(\varpi_\mathfrak{p})$.
Set $Y_1:=X_1$. Then, we have 
$$\wh{\sD}_\mathfrak{p}=Y_1\partial_{Y_1}+(\lambda Y_2+g_\mathfrak{p}(Y_1,Y_2))\partial_{Y_2},$$
where $g_\mathfrak{p}\in\varpi_\mathfrak{p}R_\mathfrak{p}[[Y_1,Y_2]]$ 
vanishes to order at least $2$ at $0$. On the other hand, recall that $\wh{V}_{K_\mathfrak{p}}$ is defined by $Y_1=\phi(Y_2)$ in $\textup{Spf}\,\wh{\sO}_{\sX_{K_\mathfrak{p}},\sP_{K_\mathfrak{p}}}\cong \textup{Spf}\,K_\mathfrak{p}[[Y_1,Y_2]]$, where $\phi$ is the unique formal power series in $K_\mathfrak{p}[[T]]$ solution of the differential equation 
$$T\phi'(T)-\alpha \phi(T) = g_\mathfrak{p}(X,\phi(X))$$
such that $\phi(0)=\phi'(0)=0$. 
Let us write $\phi=\sum_{m\ge 2}a_m T^m$
and $g_\mathfrak{p}=\sum_{m\ge 0}b_m Y_2^m$, where $b_m \in \varpi_\mathfrak{p}R_\mathfrak{p}[[Y_1]]$. Then $||c_b||_{\mathfrak{p},1} \le |\omega_\mathfrak{p}|=\frac{1}{p^{[K\mathfrak{p}:\mathbb{Q}_p]}}$, so that Proposition \ref{prop:size_solution_ode} applies to show that
the size $S_{\sX_{R_\mathfrak{p}}}(\wh{V}_{K_\mathfrak{p}})$ satisfies
$$\log \frac{1}{S_{\sX_{R_\mathfrak{p}}}(\wh{V}_{K_\mathfrak{p}})} \le C(\alpha) [K_\mathfrak{p}:\mathbb{Q}_p]\frac{(\log p)^2}{p^2},$$
where $C(\alpha)>0$ depends only on $\alpha$.

\medskip

\noindent\textit{Step 3.} In the general setting, let $Y$ be the blow-up of $X$ at $P$, and let $M$ be the foliation on $Y$ induced by $L$. The exceptional divisor $E$ is $M$-invariant and contains exactly two singularities $Q_1$ and $Q_2$ (defined over $K$) of $M$, both reduced and non-degenerate. In addition, relabeling $Q_1$ and $Q_2$ if necessary, we may assume that there is a smooth formal subscheme $\wh{V}_1$ (resp. $\wh{W}_1$) of $\wh{Y}_{Q_1}$ defined over $K$ (resp. $\wh{Y}_{Q_2}$) which is $\wh{M}$-invariant and such that the morphism $Y \to X$ induces an isomorphism
$\wh{V_1}\cong \wh{V}$ (resp. $\wh{W_1}\cong \wh{W}$). The blow-up $\sY$ of $\sX$ along $\sP$ is a smooth and quasi-projective model of $Y$ over $R[1/N]$. In addition, $Q_1$ and $Q_2$ extends to points $\sQ_1\in \sY(R[1/N])$ and $\sQ_2\in \sY(R[1/N])$, and $M$ extends to a foliation $\sM \subset T_{\sY /\textup{Spec}\, R[1/N]}$. We have a commutative diagram
\begin{center}
\begin{tikzcd}[row sep=large]
\sY \ar[r]\ar[d] & \textup{Bl}_0 \mathbb{A}^2_{R[1/N]}\ar[d]\\
\sX \ar[r] & \mathbb{A}^2_{R[1/N]},
\end{tikzcd}
\end{center}
where the horizontal arrows are \'etale morphims. Let $U_1 \cong \mathbb{A}^2_{R[1/N]} \subset \textup{Bl}_0 \mathbb{A}^2_{R[1/N]}$ be the affine charts containing $Q_1$ with coordinates $(y_1,y_2)$ (centered at $Q_1$). Then the natural morphism
$U_1 \to \mathbb{A}^2_{R[1/N]}$ maps $(y_1,y_2)$ to $(y_1y_2,y_2)$. This immediately implies that the formal subscheme $\wh{V}_1$ is defined by $Y_1=\phi(Y_2)/Y_2$
in $\wh{Y}_{Q_1}=\textup{Spf}\,\wh{\sO}_{Y, Q_1}\cong \textup{Spf}\, K[[Y_1,Y_2]]$.
Finally, recall that the size (of a smooth formal scheme) is invariant by \'etale localization. Let again $\mathfrak{p}$ be a maximal ideal of $R$, and suppose that $\mathfrak{p}\nmid N$, and that $K_\mathfrak{p}$ is absolutely unramified. Then
Lemma \ref{lemma:size_graph_2} and Step 2 imply that 
$S_{\sX_{R_\mathfrak{p}}}(\wh{V}_{K_\mathfrak{p}})$ satisfies
$$\log \frac{1}{S_{\sX_{R_\mathfrak{p}}}(\wh{V}_{K_\mathfrak{p}})} \le
\log \frac{1}{S_{\sY_{R_\mathfrak{p}}}(\wh{W}_{K_\mathfrak{p}})}
\le C'(\alpha) [K_\mathfrak{p}:\mathbb{Q}_p]\frac{(\log p)^2}{p^2}$$
where $C'(\alpha)>0$ depends only on $\alpha$. 
This immediately implies that $\wh{V}$ is $A$-analytic, completing the proof of the theorem.
\end{proof}

Let $L$ be a foliation on a smooth complex quasi-projective surface $X$, and let $P$ be a singular point of $L$. By \cite[Theorem]{camacho_sad}, there exists a (possibly singular) analytic curve passing through $P$. The following is an easy consequence of Theorem \ref{thm:A-analyticity} above.

\begin{cor}\label{cor:cor}
Let $X$ be a smooth quasi-projective surface over a number field $K$, and let $L$ be a foliation on $X$. Suppose that $L$ is closed under $p$-th powers for almost all primes $p$.
Let $P\in X(K)$ be a singular point of $L$, and let $\wh{V}$ be an $\wh{L}$-invariant smooth formal subcheme of the completion $\wh{X}_P$ of $X$ at $P$ defined over $K$, where $\wh{L}$ denotes the $($formal$)$ foliation induced by $L$ on $\wh{X}_P$. Then $\wh{V}$ is $A$-analytic.
\end{cor}

\begin{proof}
Let $\wb{K}$ be an algebraic closure of $K$. By a result proved by Seidenberg (\cite{seidenberg_reduction}), there exists a composition of a finite number of blow-ups $Y_{\wb{K}} \to X_{\wb{K}}$ of $\wb{K}$-rational points such the foliation $M_{\wb{K}}$ induced by $L$ on $Y_{\wb{K}}$ has reduced singularities. The variety $Y_{\wb{K}}$ is defined over a finite extension $F$ of $K$. Let $M$ be the foliation on $Y$ indiced by $L$.
Notice that $M$ is closed under $p$-th powers for almost all primes $p$. Then Theorem \ref{thm:A-analyticity} applied to the proper transform of $\wh{V}_F$ in $Y$ together with \cite[Proposition 3.4]{bost_acl} and Lemma \ref{lemma:size_graph_2} imply that $\wh{V}$ is A-analytic.
\end{proof}

\providecommand{\bysame}{\leavevmode\hbox to3em{\hrulefill}\thinspace}
\providecommand{\MR}{\relax\ifhmode\unskip\space\fi MR }
\providecommand{\MRhref}[2]{%
  \href{http://www.ams.org/mathscinet-getitem?mr=#1}{#2}
}
\providecommand{\href}[2]{#2}


\end{document}